\documentclass[a4paper]{article}

\usepackage[utf8]{inputenc}
\usepackage[T1]{fontenc}
\usepackage{geometry}
\usepackage[english]{babel}
                              
\usepackage{amsmath,amssymb,amsthm}
\usepackage{mathrsfs}
\usepackage{dsfont}
\usepackage{color}
\usepackage{hyperref}
\usepackage{graphicx}
\usepackage{ulem}
\usepackage[dvipsnames]{xcolor}
\usepackage{cancel}

\newtheorem{thm}{Theorem}[section]
\newtheorem{lem}[thm]{Lemma}
\newtheorem{cor}[thm]{Corollary}
\theoremstyle{definition}
\newtheorem{dfn}[thm]{Definition}
\theoremstyle{remark}
\newtheorem{rmk}{Remark}
\newtheorem{example}{Example}

\newcommand{\N}{\mathbb{N}}
\newcommand{\R}{\mathbb{R}}

\newcommand{\bOm}{\mathbf{\Omega}}
\newcommand{\bK}{\mathbf{K}}

\newcommand{\bB}{\mathbf{B}}

\newcommand{\bX}{\mathbf{X}}

\newcommand{\mL}{\mathrm{L}}

\newcommand{\mH}{\mathrm{H}}

\newcommand{\vx}{\mathbf{x}}
\newcommand{\vy}{\mathbf{y}}
\newcommand{\ve}{\mathbf{e}}

\newcommand{\vk}{\mathbf{k}}

\newcommand{\vn}{\mathbf{n}}
\newcommand{\vO}{\mathbf{0}}

\newcommand{\vu}{\mathbf{u}}
\newcommand{\vv}{\mathbf{v}}

\newcommand{\cM}{\mathcal{M}}
\newcommand{\cC}{\mathcal{C}}

\newcommand{\Beta}{\mathrm{B}}

\newcommand{\md}{\mathrm{d}}
\newcommand{\ind}{\mathds{1}}

\renewcommand{\div}{\mathop{\mathrm{div}}}
\newcommand{\grad}{\mathop{\mathrm{grad}}}
\newcommand{\e}{\mathrm{e}}

\newcommand{\norm}[1]{{\left\vert\kern-0.25ex\left\vert\kern-0.25ex\left\vert #1 
    \right\vert\kern-0.25ex\right\vert\kern-0.25ex\right\vert}}

\renewcommand{\emph}[1]{{\it #1}}

\title{Stokes, Gibbs and volume computation of semi-algebraic sets}
\author{Matteo Tacchi$^{1}$, Jean Bernard Lasserre$^{2,3}$, Didier Henrion$^{2,4}$}

\begin{document}

\maketitle
                            
\footnotetext[1]{Laboratoire d'Automatique, École Polytechnique Fédérale de Lausanne (EPFL), Switzerland.}
\footnotetext[2]{LAAS-CNRS, 7 avenue du colonel Roche, 31400 Toulouse, France.}
\footnotetext[3]{Institut de Math\'ematiques de Toulouse,  Universit\'e de Toulouse, France.}
\footnotetext[4]{Faculty of Electrical Engineering, Czech Technical University in Prague, Czechia.}

 \centerline{Corresponding author: Matteo Tacchi, \texttt{matteo.tacchi@epfl.ch}}

\begin{abstract}
We consider the problem of computing the Lebesgue volume of compact basic semi-algebraic sets. In full generality, it can be 
approximated as closely as desired by a converging hierarchy of upper bounds obtained by applying the 
Moment-SOS (sums of squares) methodology to a certain infinite-dimensional linear program (LP). At each step one solves a semidefinite relaxation of the LP which involves 
pseudo-moments up to a certain degree. Its dual computes a polynomial of same degree
which approximates from above the discontinuous indicator function of the set, hence with a typical Gibbs phenomenon
which results in a slow convergence of the associated numerical scheme. Drastic improvements have been observed by introducing in the initial LP
additional linear moment constraints obtained from a certain application of Stokes' theorem for integration
on the set. However and so far there was no rationale to explain this behavior. We provide a refined version of this 
extended LP formulation. When the set is the smooth super-level set of a single polynomial, we show that the
dual of this refined LP has an optimal solution which is a continuous function.
Therefore in this dual
one now approximates a continuous function
by a polynomial, hence with no Gibbs phenomenon, which explains and improves the already observed drastic acceleration of the convergence 
of the hierarchy. Interestingly, the technique of proof involves {recent} results on Poisson's partial differential equation (PDE).
\end{abstract}

\small
\begin{center}
\textbf{Keywords}
\end{center}

Numerical methods for multivariate integration; real algebraic geometry; convex optimization; Stokes' theorem; Gibbs phenomenon

\begin{center}
\textbf{Acknowledgments}
\end{center}

This work benefited from discussions with Swann Marx.

The work of M. Tacchi was funded by the French company R\'{e}seau de Transport d'\'{E}lectricit\'{e}, as well as the Swiss National Science Foundation under the “NCCR Automation” grant n$^{\circ}$51NF40\_{}180545.

The work of J.B. Lasserre was partly funded by the AI Interdisciplinary Institute ANITI through the French “Investing for the Future PIA3” program under the Grant agreement n$^{\circ}$ANR-19-PI3A-0004.

 \normalsize

\section{Introduction}

Consider the problem of computing the Lebesgue volume {$\lambda(\bK)$} of a compact basic semi-algebraic set 
$\bK\subset\R^n$. For simplicity of exposition we will restrict to the case where
$\bK$ is the {smooth} super-level set $\{\,\vx: g(\vx)\geq0\,\}\subset\R^n$ of a single polynomial
$g$.

If $\bK$ is a {convex body} then several 
procedures are available; see e.g. exact deterministic methods for convex polytopes \cite{vinci}, or non deterministic {Hit-and-Run} methods \cite{smith1,smith2} and the more recent \cite{vempala1,vempala2}.
Even {approximating} $\lambda(\bK)$ by deterministic methods 
is still a hard problem as explained in e.g. \cite{vempala2} and references therein. 
In full generality with no specific assumption on $\bK$ such as convexity, the only general method available is Monte-Carlo, that is, one samples $N$ points
according to Lebesgue measure $\lambda$ normalized on a simple set $\bB$ (e.g. a box or an ellipsoid) that contains $\bK$.
If $\rho_N$ is the proportion of points that fall into $\bK$ then the random variable 
$\rho_N\:\lambda(\bB)$ provides a good estimator of $\lambda(\bK)$
with convergence guarantees as $N$ increases. However this estimator is non deterministic and neither provides
a lower bound nor an upper bound on $\lambda(\bK)$. 

When $\bK$ is a compact basic semi-algebraic set,
a deterministic numerical scheme described in \cite{hls09} provides a sequence $(\tau_k)_{k\in\N} \subset \R$ of upper bounds that converges to $\lambda(\bK)$ as $k$ increases. Briefly,
\begin{eqnarray}
\label{intro:eq-dual}
\lambda(\bK)&=&\displaystyle\inf_{p\in\R[\vx]}\,\left\{\,\int p\,d\lambda \::\: p\geq \ind_\bK\mbox{ on $\bB$}\,\right\}\\
\label{intro:eq-dual-k}
\tau_k &=&\displaystyle\inf_{p\in\R[\vx]_k}\,\left\{\,\int p\,d\lambda \::\: p\geq \ind_\bK\mbox{ on $\bB$}\,\right\},
\end{eqnarray}
with $\vx\mapsto \ind_\bK(\vx)=1$ if $\vx\in\bK$ and $0$ otherwise. One can notice that minimizing sequences for \eqref{intro:eq-dual} and \eqref{intro:eq-dual-k} also minimize the $L^1(\bB,\lambda)$-norm $\Vert p - \ind_\bK \Vert_1$ (with convergence to $0$ in the case \eqref{intro:eq-dual}). 
As the upper bound $\tau_k> \lambda(\bK)$ is obtained by restricting the search in \eqref{intro:eq-dual-k} to polynomials of degree at most $k$,
the infimum is attained 
and an optimal solution can be obtained by
solving a semidefinite program. Of course, 
the size of the resulting semidefinite program increases with the degree $k${: this is the so-called Moment-SOS hierarchy};
for more details the interested reader is referred to \cite{hls09}.

{ Also focusing on compact semi-algebraic sets, \cite{moab} proposes a symbolic method to compute the volume of $\bK$ with absolute precision $2^{-p}$, in time $O(p\log (p)^{3+\varepsilon})$ for any $\varepsilon > 0$ as $p \to \infty$. This is in sharp contrast with the approach considered here, which consists in approximating problem \eqref{intro:eq-dual} with the sequence of problems \eqref{intro:eq-dual-k} indexed by $k$. Indeed,
\begin{itemize}
\item In \cite{hls09}, for any $k \in \N$, $\tau_k$ is guaranteed to be a converging upper bound for $\lambda(\bK)$, i.e. $\tau_k - \lambda(\bK) > 0$, while  \cite{moab} also guarantees convergence of the approximant to $\lambda(\bK)$ but gives no information on the sign of the difference between the two quantities.
\item \cite{moab} uses symbolic computations that can achieve arbitrary precision, while \cite{hls09} uses numerical computations based on semidefinite programming, limited to floating-point arithmetic precision.
\item The approach of \cite{moab} can be used  to approximate other quantities than the volume, namely real periods of algebraic surfaces. The approach of \cite{hls09} was extended to approximate sets relevant in systems control, such as regions of attraction or maximal positively invariant sets (see e.g. \cite{hk14}). In this context, the present contribution can help improving the Moment-SOS hierarchy for assessing the stability of polynomial differential systems.
\end{itemize}

When solving problem \eqref{intro:eq-dual-k}}, clearly a Gibbs phenomenon\footnote{
The Gibbs phenomenon appears at a jump discontinuity when one numerically approximates a piecewise $C^1$ function with a polynomial function, e.g. by its Fourier series; see e.g. \cite[Chapter 9]{t13}.}
takes place as one tries to approximate on $\bB$ and from above,
the discontinuous function $\ind_\bK$ by a polynomial of degree at most $k$. This makes 
the convergence of the upper bounds $\tau_k$ very slow (even for modest dimension problems). A trick was used 
in \cite{hls09} to accelerate this convergence but at the price of loosing monotonicity of the resulting sequence.

In fact \eqref{intro:eq-dual} is a dual of the following infinite-dimensional Linear program (LP) on measures
\begin{equation}
\label{intro:eq-primal}
\sup_{\mu} \;\{\,\mu(\bK) \::\:\mu\,\leq\,\lambda\,;\: \mu\in \cM(\bK)_+\,\}
\end{equation}
(where $\cM(\bK)_+$ is the space of finite Borel measures on $\bK$). Its optimal value is also $\lambda(\bK)$
and is attained at the unique optimal solution $\mu^\star := \lambda_\bK { = \ind_\bK \, \lambda}$ (the restriction of $\lambda$ to $\bK$).

{\bf A simple but key observation.} As one knows the unique optimal solution $\mu^\star=\lambda_{\bK}$ of \eqref{intro:eq-primal},
{any} constraint satisfied by $\mu^\star$ 
(in particular, linear constraints) can be included as a constraint on $\mu$ in \eqref{intro:eq-primal} without changing 
the optimal value and the optimal solution. While these constraints 
provide additional {restrictions} in \eqref{intro:eq-primal}, they
translate into additional {degrees of freedom} in the dual 
(hence a {relaxed} version of \eqref{intro:eq-dual}), and therefore
better approximations when passing to the finite-dimensional relaxed version of \eqref{intro:eq-dual-k}. A first set 
of such linear constraints experimented in \cite{l17} and later in 
\cite{lm20}, resulted in drastic improvements but with no clear rationale behind
such improvements.

{\bf Contribution.}
The main message and result of this paper is that {there is an appropriate 
set of additional linear constraints on $\mu$ in \eqref{intro:eq-primal} 
such that the resulting dual (a relaxed version of \eqref{intro:eq-dual}) has an explicit {continuous} optimal solution 
with value $\lambda(\bK)$. These additional linear contraints (called Stokes constraints) come from
an appropriate modelling of Stokes' theorem for integration over $\bK$,
a refined version 
of that in \cite{l17}.
Therefore the optimal continuous solution can be approximated efficiently
by polynomials with no Gibbs phenomenon, by the hierarchy of semidefinite relaxations defined in \cite{hls09} (adapted to these new linear constraints).} 
Interestingly, the technique of proof and the construction of the optimal solution invoke results from the field of elliptic partial differential equations (PDE), namely {a recent extension of standard Schauder estimates from Dirichlet problems to Neumann formulations}.

{\bf Outline.} In Section \ref{sec:volume} we recall the primal-dual linear formulation of the volume problem, and we explain why the dual value is not attained, which results in a Gibbs phenomenon. In Section \ref{sec:stokes} we revisit the acceleration strategy based on Stokes' theorem, with the aim of introducing in Section \ref{sec:newstokes} a more general acceleration strategy and a new primal-dual linear formulation of the volume problem. Our main result, attainment of the dual value in this new formulation, is stated as Theorem \ref{main} at the end of Section \ref{sec:newstokes}. The drastic improvement in the convergence to $\lambda(\bK)$ 
is illustrated on {various} simple example{s}.

\section{Linear reformulation of the volume problem}\label{sec:volume}

Consider a compact basic semi-algebraic set
$$ \bK := \{ \vx \in \R^n : g(\vx) \geq 0 \} $$
with $g \in \R[\vx]$. We suppose that $\bK \subset \bB$ where $\bB$ is a compact basic semi-algebraic set for which we know the moments $\int_\bB \vx^\vk \; \md\vx$ of the Lebesgue measure $\lambda_\bB$, where $\vx^\vk := x^{k_1}_{1} x^{k_2}_{2} \cdots x^{k_n}_{n}$ denotes a multivariate monomial of degree $\vk \in \N^n$. We assume that 
$$\bOm := \{\vx \in \R^n : g(\vx) > 0\}$$
is a nonempty open set with closure\footnote{Whereas $\bK$ is a rather standard notation for a compact semi-algebraic set in polynomial optimization, the addition of the notation $\bOm$ is motivated by the conventions used for open sets in the differential geometry and PDE analysis literature. To account for both uses, we denote by $\bOm$ the interior of $\bK$.} $$\overline{\bOm} = \bK,$$
and that  its boundary $$\partial \bOm = \partial \bK = \bK \setminus \bOm$$ is $C^1$ in the sense that it is locally the graph of a continuously differentiable function.
We want to compute the Lebesgue volume of $\bK$, i.e., the mass of the Lebesgue measure $\lambda_\bK$:
$$\lambda(\bK) := \int_{\bK} d\vx = \int_{\R^n} d\lambda_\bK(\vx).$$

If $\bX \subset \R^n$ is a compact set, denote by $\cM(\bX)$ the space of signed Borel measures on $\bX$, which identifies with the topological dual of $C^0(\bX)$, the space of continuous functions on $\bX$.
Denote by $\cM(\bX)_+$ the convex cone of non-negative Borel measures on $\bX$, and by  $C^0(\bX)_+$ the convex cone of non-negative continuous functions on $\bX$.

In \cite{hls09} a sequence of upper bounds converging to $\lambda(\bK)$ is obtained by applying the Moment-SOS hierarchy 
\cite{l10} (a family of finite-dimensional convex relaxations) to approximate as closely as desired
the (primal) infinite-dimensional LP on measures:
\begin{align}
\max_{\mu} \; & \mu(\bK) \label{pl:primal} \\
\mathrm{s.t.} \; & \mu \in \cM(\bK)_+ \nonumber \\
& \lambda_\bB - \mu \in \cM(\bB)_+ \nonumber
\end{align}
whose optimal value is $\lambda(\bK)$, attained for $\mu^\star := \lambda_\bK$. 
The LP \eqref{pl:primal} has an infinite-dimensional LP dual on {continuous} functions which reads:
\begin{align}
\inf_{w} \; & \int_\bB w \; d\lambda \label{pl:dual} \\
s.t. \; & w \in C^0(\bB)_+ \nonumber \\
&  w\vert_\bK - 1 \in C^0(\bK)_+. \nonumber
\end{align}
Observe that \eqref{pl:dual} consists of approximating the discontinuous indicator function $\ind_\bK$ (equal to one on $\bK$ and zero elsewhere) from above by continuous functions $w$, in minimizing the $L^1(\bB)$-norm $\Vert w-\ind_{\bK}\Vert_1$. Clearly
the infimum $\lambda(\bK)$ is {not} attained.

Since $\bK$ is generated by a polynomial $g$, and measures on compact sets are uniquely determined by their moments,
one may apply the Moment-SOS hierarchy \cite{l10} for solving \eqref{pl:primal}. The moment relaxation of
\eqref{pl:primal} consists of replacing $\mu$ by finitely many of its moments $\vy$, say up to degree $d \in \N$. Then the cone of moments is relaxed by a linear slice of the semidefinite cone constructed from so-called moment and localizing matrices indexed by $d$, as defined in e.g. \cite{l10}, and which defines a semidefinite program. Therefore the dual of this 
semidefinite program (i.e., the dual SOS-hierarchy) is a {strengthening} of \eqref{pl:dual} where

(i) continuous functions $w$ are replaced with polynomials of increasing degree $d$,
and

(ii) nonnegativity constraints are replaced with Putinar's SOS-based certificates of positivity \cite{p93} which translate to
 semidefinite constraints on the coefficients of polynomials; again the interested reader is referred to \cite{hls09,l10} for more details.
 
 For each fixed degree $d$, a valid upper bound on $\lambda(\bK)$ is computed by solving a primal-dual pair of convex semidefinite programming problems (not described here). As proved in \cite{hls09} by combining Stone-Weierstrass' theorem and Putinar's Positivstellensatz \cite{p93},
 
(i) there is no duality gap between each primal semidefinite relaxation of the hierarchy and its dual, and

(ii) the 
resulting sequence of upper bounds converges to $\lambda(\bK)$ as $d$ increases.

The main drawback of this numerical scheme is its typical slow convergence, observed already for very simple univariate examples, see e.g. \cite[Figs. 4.1 and 4.5]{hls09}. The best available theoretical convergence speed estimates are also pessimistic, with an asymptoptic rate of $\log\log d$ \cite{kh18}. Slow convergence is mostly due to the so-called Gibbs phenomenon which is well-known in numerical analysis \cite[Chapter 9]{t13}.
Indeed, as already mentioned, solving \eqref{pl:dual} numerically amounts to approximating the discontinuous function $\ind_\bK$ from above with polynomials of increasing degree, which generates oscillations and overshoots and slows down the convergence, see e.g. \cite[Figs. 4.2, 4.4, 4.6, 4.7, 4.10, 4.12]{hls09}.

\begin{example}\label{ex:gibbs}
Let $\bK:=[0,1/2] \subset \bB:=[-1,1]$. In Figure \ref{fig:gibbs} is displayed the degree-10 and degree-20 polynomials $w$ obtained by solving the dual of SOS {strengthenings} of problem \eqref{pl:primal}. We can clearly see  bumps, typical of a Gibbs phenomenon at points
of discontinuity.
\begin{figure}[h]\centering
\includegraphics[width=0.4\textwidth]{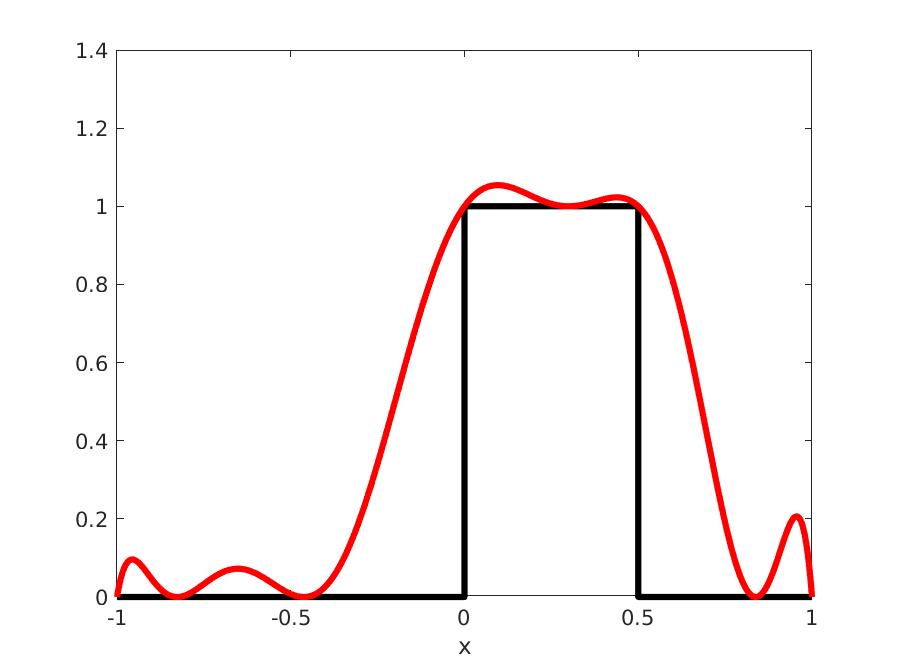}
\includegraphics[width=0.4\textwidth]{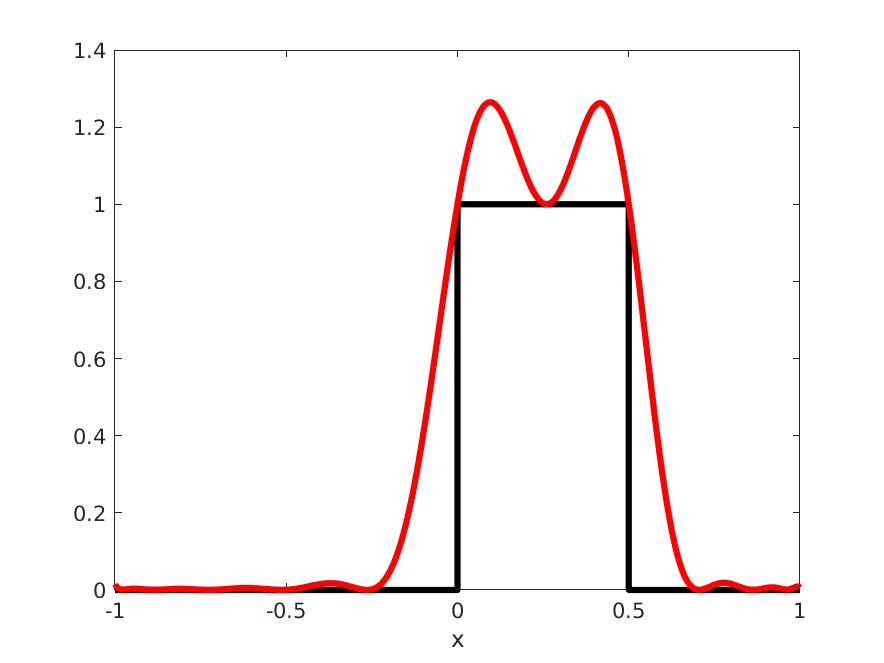}
\caption{Gibbs effect occurring when approximating from above with a polynomial of degree 10 (left red curve) and 20 (right red curve) the indicator function of an interval (black curve).\label{fig:gibbs}}
\end{figure}
\end{example}

An idea to bypass this limitation consists of adding certain linear constraints to the finite-dimensional semidefinite relaxations,
to make their optimal values larger and so closer to the optimal value $\lambda(\bK)$.
Such linear constraints must be chosen appropriately:

(i) they must be {redundant} for the infinite-dimensional moment LP
on measures \eqref{pl:primal}, and

(ii) become {active} for its finite-dimensional relaxations. 

This is the heuristic proposed in \cite{l17} to accelerate the Moment-SOS hierarchy for evaluating transcendental integrals on semi-algebraic sets. 
These additional linear constraints on the moments $\vy$ of $\mu^\star$ are obtained from
an application of Stokes' theorem for integration on $\bK$, a classical result in differential geometry.
It has been also observed experimentally that this heuristic accelerates significantly the convergence of the hierarchy 
in other applied contexts, e.g.  in chance-constrained optimization problems \cite{w18}.

\section{Introducing Stokes constraints} \label{sec:stokes}

In this section we explain the heuristic introduced in \cite{l17} to accelerate convergence of the Moment-SOS hierarchy by adding linear constraints on the moments of $\mu^\star$. These linear constraints are obtained  from a certain application of Stokes' theorem for integration on $\bK$.

\subsection{Stokes' Theorem and its variants}

\begin{thm}[Stokes' Theorem]
Let $\bOm \subset \R^n$ be a $C^1$ open set\footnote{An open set is said to be $C^1$ if its boundary is locally the graph of a $C^1$ function (up to reordering coordinates and changing orientation, see e.g.~\cite[Section C.1.]{e10}).} {with closure $\bK$}. For any $(n-1)$-differential form $\omega$ on ${\bK}$, it holds $$\displaystyle\int_{\partial\bOm} \omega = \int_\bOm d\omega.$$
\end{thm}

\begin{cor}
In particular, for $\vu \in C^1({\bK})^n$ and $\omega(\vx) = \vu(\vx) \cdot \vn_\bOm(\vx) \; d\sigma(\vx)$, where the dot is the inner product, $\sigma$ is the surface or Hausdorff measure on $\partial \bOm$ and $\vn_\bOm$ is the outward pointing normal to $\partial \bOm$, we obtain the Gauss formula
\begin{equation} \label{eq:gaussform}
\int_{\partial\bOm} \vu(\vx) \cdot \vn_\bOm(\vx) \; d\sigma(\vx) = \int_\bOm \div \vu(\vx) \; d\vx.
\end{equation}
With the choice $\vu(\vx) := u(\vx) \ \ve_i$ where $u \in C^1({\bK})$ and $\ve_i$ is the vector of $\R^n$ with one at entry $i$ and zeros elsewhere, for $i=1,\ldots,n$, we obtain the dual Gauss formula
\begin{equation} \label{eq:stokesthm}
\int_{\partial\bOm} u(\vx) \ \vn_\bOm(\vx) \; d\sigma(\vx) = \int_\bOm \grad u(\vx) \; d\vx.
\end{equation}
\end{cor}

\begin{proof}
These are all particular cases of \cite[Theorem 6.10.2]{hb07}.
\end{proof}

\subsection{Original Stokes constraints}

Associated to a sequence $\vy=(y_\vk)_{\vk \in \N^n} \in \R^{\N^n}$, introduce the Riesz linear functional $\mL_\vy : \R[\vx] \rightarrow \R$ which acts on a polynomial $p := \sum_{\vk} p_\vk \ \vx^\vk \in \R[\vx]$ by $\mL_\vy(p) := \sum_{\vk} p_\vk \ y_\vk$.
Thus, if $\vy$ is the sequence of moments of $\lambda_\bK$, i.e. $y_\vk:=\int_\bK \vx^\vk d\vx$ for all $\vk \in \N^n$, then $\mL_\vy(p) = \int_\bK p(\vx) d\vx$ and by \eqref{eq:stokesthm} with $u(\vx) := \vx^\vk g(\vx)$:
\begin{align*}
\mL_{\vy}(\grad(\vx^\vk g)) \ & = \int_\bK \grad(\vx^\vk g(\vx)) \; d\vx \\
& = \int_{\partial \bK} \vx^\vk g(\vx) \ \vn_\bK(\vx) \; d\sigma(\vx) \,=\,0,
\end{align*}
since by construction $g$ vanishes on $\partial \bK$. Thus while in the infinite-dimensional LP \eqref{pl:primal} one may add the linear constraints 
\[\int _{\bK}\grad(\vx^\vk g)\,d\mu\,=\,0\,\qquad\forall \vk\in\N^n,\]
without changing its optimal value $\lambda(\bK)$, on the other hand inclusion of the linear moment constraints
\begin{equation} \label{eq:momstokes}
\mL_\vy(\grad(\vx^\vk g)) = 0\,,\quad \vert\vk\vert\leq 2d+1-{\rm deg}(g)
\end{equation}
in the moment relaxation with pseudo-moments $\vy$ of degree at most $d$, will decrease the optimal value
of the initial relaxation.

In practice, it was observed that adding constraints \eqref{eq:momstokes} dramatically speeds up the convergence of the Moment-SOS hierarchy, see e.g. \cite{l17,w18}. One main goal of this paper is to provide a qualitative mathematical rationale behind
this phenomenon.

\subsection{Infinite-dimensional Stokes constraints}

In \cite{twlh19}, Stokes constraints were formulated in the infinite-dimensional setting, and a dual formulation was obtained in the context of the volume problem. Using \eqref{eq:gaussform} with $\vu = g \vv$ (which vanishes on $\partial \bK$) and $\vv \in C^1(\bK)^n$ arbitrary, yields:
$$ \int_\bK (\grad g(\vx)  \cdot \vv(\vx) + g(\vx) \div \vv(\vx)) \; \md\vx = \int_{\partial\bK}g \vv \ \vn_\bK \ d\sigma = 0\,,$$
which can be written equivalently (in the sense of distributions) as $$(\grad g) \lambda_\bK - \grad(g\lambda_\bK) = 0\,.$$
This allows to rewrite problem \eqref{pl:primal} as
\begin{align}
\max_\mu \; & \mu(\bK) \label{pl:primstokes} \\
s.t. \; & \mu \in \cM(\bK)_+ \nonumber \\
& \lambda_\bB - \mu \in \cM(\bB)_+ \nonumber \\
& (\grad g) \mu - \grad (g \mu) = 0 \nonumber
\end{align}
without changing its optimal value $\lambda(\bK)$ attained at $\mu^\star = \lambda_\bK$.

Using infinite-dimensional convex duality as in e.g. the proof of Theorem 2 in \cite{hk14},
the dual of LP \eqref{pl:primstokes} reads
\begin{align}
\inf_{\vv,w} \; & \int_\bB w \; d\lambda \label{pl:dualstokes} \\
s.t. \; & \vv \in C^1(\bK)^n \nonumber \\
& w \in C^0(\bB)_+ \nonumber \\ 
& w\vert_\bK - \div(g\vv) - 1 \in C^0(\bK)_+. \nonumber
\end{align}
{\bf Crucial observation.} Notice that $w$ in {\eqref{pl:dualstokes}} is {not} required to approximate $\ind_\bK$ from above anymore. Instead, it should approximate $1 + \div(g\vv)$ on $\bK$ and $0$ outside $\bK$. Hence, provided that $1+\div(g\vv) = 0$ on $\partial\bK$, $w$ might be a continuous function for some well-chosen $\vv \in C^1(\bK)^n$, and therefore an optimal solution 
of \eqref{pl:dualstokes} (i.e., the infimum is a minimum). As a result, the Gibbs phenomenon would disappear
and convergence would be faster. 

{The issue is then to determine whether the infimum in \eqref{pl:dualstokes} is attained or not. And if not, are there 
other special features of problem \eqref{pl:dualstokes} that can be exploited to yield
more efficient semidefinite relaxations ?}

\section{New Stokes constraints and main result}\label{sec:newstokes}

In the previous section, the Stokes constraint
$$ \int_\bK (\vv(\vx) \cdot \grad g(\vx) + g(\vx) \div \vv(\vx))\; d\mu(\vx) = 0$$
or equivalently (in the sense of distributions)
\begin{equation}\label{oldstokes}
(\grad g) \mu - \grad (g \mu) = 0
\end{equation}
(with $\mu \in \cM(\bK)_+$ being the Lebesgue measure on $\bK$)
was obtained as a particular case of Stokes' theorem with $\vu=g\vv$ in \eqref{eq:gaussform}. 
Instead, we can use a more general version with $\vu$ not in factored form, and also use the fact that $\forall \vx \in \partial \bK$, $0 \neq \grad g(\vx) = - |\grad g(\vx)| \ \vn_\bK(\vx)$ (here $|\vy| := \sqrt{\vy\cdot\vy}$ is the $n$-dimensional Euclidean norm), to obtain
$$\int_{\bK} \div \vu(\vx)\; d\mu(\vx) = - \int_{\partial \bK} \vu(\vx) \cdot \grad g(\vx)\; d\nu(\vx)\,,$$
or equivalently (in the sense of distributions)
\begin{equation}\label{newstokes}
\grad \mu = (\grad g) \nu\,,
\end{equation}
with $\mu \in \cM(\bK)_+$ being the Lebesgue measure on $\bK$ and $\nu \in \cM(\partial\bK)_+$ 
being the measure having density $1/|\grad g(\vx)|$ with respect to the $(n-1)$-dimensional Haussdorff measure $\sigma$ on $\partial\bK$.
The same linear equation was used in \cite{lm20} to compute moments of the Hausdorff measure.
In fact, equation \eqref{newstokes} is a generalization of equation \eqref{oldstokes} in the following sense.

\begin{lem}
If $\nu \in \cM(\partial\bK)_+$ is such that $\mu \in \cM(\bK)_+$ satisfies \eqref{newstokes}, then $\mu$ also satisfies \eqref{oldstokes}.
\end{lem}
\begin{proof}
Equation \eqref{newstokes} means that $\int_{\bK} \div \vu(\vx) \; d\mu(\vx) + \int_{\partial \bK} \vu(\vx) \cdot \grad g(\vx) \; d\nu(\vx)  = 0$ for all $\vu \in C^1(\bK)^n$. In particular if $\vu = g\vv$ for some $\vv \in C^1(\bK)^n$ then \eqref{newstokes} reads
\[\int_\bK (\vv(\vx) \cdot \grad g(\vx) + g(x) \div \vv(\vx))\; d\mu(\vx) = 0\,,\]
which is precisely \eqref{oldstokes}.
\end{proof}

Hence we can incorporate linear constraints  \eqref{newstokes} on $\mu$ and $\nu$, to rewrite problem \eqref{pl:primal} as
\begin{align}
\max_{\mu,\nu} \; & \mu(\bK) \label{lp:jean_primal}\\
s.t. \; & \mu \in \cM(\bK)_+ \nonumber \\
& \nu \in \cM(\partial\bK)_+ \nonumber \\
& \lambda_\bB - \mu \in \cM(\bB)_+ \nonumber \\
& (\grad g) \nu - \grad \mu = 0 \nonumber
\end{align}
without changing its optimal value $\lambda(\bK)$ attained at $\mu^\star = \lambda_\bK$
and $\nu^\star = \sigma/|\grad g|$. Notice that LP \eqref{lp:jean_primal} involves {two} measures $\mu$ and $\nu$ whereas
LP \eqref{pl:primstokes} involves only one measure $\mu$.

Next, by convex duality as in e.g. the proof of Theorem 2 in \cite{hk14},
the dual of \eqref{lp:jean_primal} reads
\begin{align}
\inf_{\vu,w} \; & \int_\bB w \,d\lambda \label{lp:jean_dual} \\
s.t. \; & \vu \in C^1(\bK)^n \nonumber \\
& w \in C^0(\bB)_+ \nonumber \\
& w\vert_\bK - \div \vu - 1 \in C^0(\bK)_+ \nonumber \\
& -(\vu \cdot \grad g)|_{\partial\bK} \in C^0(\partial \bK)_+. \nonumber
\end{align}
Our main result states that the optimal value of the dual \eqref{lp:jean_dual} is attained at some continuous function
$(w,\vu)\in C^0(\bB)_+\times C^1(\bK)^n$. Therefore, in contrast with problem \eqref{pl:dual}, there is no Gibbs phenomenon at an optimal solution of the (finite-dimensional) semidefinite {strengthening} associated 
with \eqref{lp:jean_dual}.

Let $\bOm_i$, $i=1,\ldots,N$ denote the connected components of $\bOm$, and let $$m_{\bOm_i}(g) := \frac{1}{\lambda(\bOm_i)}\int_{\bOm_i} g \; d\lambda.$$
 
\begin{thm}\label{main}
In dual LP \eqref{lp:jean_dual} the infimum is a minimum, attained at
$$w^\star(\vx) := g(\vx)\sum_{i=1}^N\frac{\ind_{\bOm_i}(\vx)}{m_{\bOm_i}(g)},\quad \vx \in \bB\,,$$
and
$$\vu^\star(\vx) := \grad u(\vx)\,,$$
where $u$ solves the Poisson PDE
$$\left\{\begin{array}{rcll}
-\Delta u(\vx) &=& 1 - w^\star(\vx),& \quad \vx \in \bOm \\
\partial_\vn u(\vx) &=& 0,& \quad \vx \in \partial\bOm.
\end{array}\right.$$
\end{thm}

\begin{rmk} The Moment-SOS hierarchy associated to LPs \eqref{lp:jean_primal} and \eqref{lp:jean_dual} yields upper bounds for the volume. Theorem \ref{main} is designed for these LPs but it has a straightforward counterpart for lower bound volume computation, obtained by replacing $\bK$ with $\bB \setminus \bOm$ in the previous developments, i.e. computing upper bounds of $\lambda(\bB\setminus\bOm)$. However, two additional technicalities should then be considered:
\begin{itemize}
\item This work only deals with semi-algebraic sets defined by a single polynomial; actually, it immediately generalizes to finite intersections of such semi-algebraic sets, as long as their boundaries do not intersect (i.e. here $\bK$ should be included in the \textit{interior} of $\bB$): the constraints on boundaries should just be splitted between the boundaries of the intersected sets.
\item This work heavily relies on the fact that the boundary of the considered set should be smooth; for this reason, computing lower bounds of the volume implies that one chooses a smooth bounding box $\bB$ (typically a euclidean ball, ellipsoid or $\ell^p$ ball), which rules out simple sets like the hypercube $[-1,1]^n$.
\end{itemize}

Upon taking into account these technicalities, Theorem \ref{main} still holds, allowing to deterministically compute \textit{upper and lower} bounds for the volume, with arbitrary precision. Of course in practice, one is limited by the performance of state-of-art SDP solvers.
\end{rmk}

\section{Proof of main result} \label{sec:proof}

Theorem \ref{main} is proved in several steps as follows:
\begin{itemize}
\item we show that the optimal dual solution satisfies a Poisson PDE;
\item we study the Poisson PDE on a union of connected domains;
\item we construct an explicit optimum for problem \eqref{lp:jean_dual}.
\end{itemize}

\subsection{Equivalence to a Poisson PDE} \label{sec:pdeq}

\begin{lem}\label{equiv}
Problem \eqref{lp:jean_dual} has an optimal solution iff there exist $\vu \in C^1({\bK})^n$ and $h \in C^0({\bK})_+$ solving
\begin{subequations}\label{eq:optimum}
\begin{align}
& h = 0 && \text{ on } \partial \bOm, \label{eq:recol} \\
& - \div\vu = 1 - h && \text{ in } \bOm, \label{eq:ediv} \\
& \vu \cdot \vn_\bOm = 0 && \text{ on } \partial\bOm. \label{eq:boundary}
\end{align}
\end{subequations}
\end{lem}
\begin{proof}
Let $(\vu,h)$ solve \eqref{eq:optimum}. Using \eqref{eq:recol}, one can define
$$w(\vx) = \left\{\begin{array}{ll}h(\vx) & \text{ if } \vx \in {\bK} \\ 0 & \text{ if } \vx \in \bB\setminus{\bK}. \end{array} \right. $$
Then $(\vu,w)$ is feasible for \eqref{lp:jean_dual} and one has
\begin{align*}
\int_\bB w \; d\lambda \; & = \int_\bOm h \; d\lambda \\
& \stackrel{\eqref{eq:ediv}}{=} \int_\bOm (1 + \div\vu) \; d\lambda \\
& \stackrel{\eqref{eq:gaussform}}{=} \lambda(\bOm) + \int_{\partial \bOm}\vu \cdot \vn_\bOm \; d\sigma \\
& \stackrel{\eqref{eq:boundary}}{=} \lambda(\bOm)
\end{align*}
so that $(\vu,w)$ is optimal.

Conversely, let $(\vu,w)$ be an optimal solution of problem \eqref{lp:jean_dual}. We know that $(\mu^\star,\nu^\star) = \left(\lambda_\bOm,\sigma/{|\grad g|}\right)$ is optimal for problem \eqref{lp:jean_primal}. Then, {the KKT optimality conditions ensure complementary slackness}:
\begin{subequations}
\begin{align}
&\displaystyle\int_\bOm (w|_\bOm - \div \vu - 1) \; d\lambda = 0, \label{eq:satmu} \\
&\displaystyle\int_{\partial \bOm} \vu \cdot \frac{\grad g}{|\grad g|} \; d\sigma = 0. \label{eq:satnu}
\end{align}
\end{subequations}
Since $w|_\bOm - \div \vu - 1$ is nonnegative, \eqref{eq:satmu} yields \eqref{eq:ediv} with $h := w|_\bOm$. Likewise, since $-(\vu \cdot \grad g)|_{\partial\bOm}$ is nonnegative, \eqref{eq:satnu} yields \eqref{eq:boundary} and thus, using \eqref{eq:gaussform}, it holds $\int_\bOm \div \vu \; d\lambda = 0$. Eventually, \eqref{eq:satmu} yields $\int_\bOm w \; d\lambda = \lambda(\bOm) = \int_\bB w \; d\lambda$ by optimality of $w$, so that $\int_{\bB \setminus \bOm} w \; d\lambda = 0$ and, since $w$ is nonnegative, $w|_{\bB \setminus \bOm} = 0$. Continuity of $w$ finally allows to conclude that $w = 0$ on $\partial \bOm$, which is exactly \eqref{eq:recol}.
\end{proof}

From Lemma \ref{equiv}, existence of an optimum for \eqref{lp:jean_dual} is then equivalent to existence of a solution to \eqref{eq:optimum}, which we rephrase as follows, defining $f := 1 - h$ and $\vu = \grad u$ with $u \in C^2({\bK})$,
and where $\Delta u := \div \grad u$ is the Laplacian of $u$, and $\partial_\vn u := \grad u \cdot \vn_\bOm$. 

\begin{lem} \label{lem:optim}
If there exist $u \in C^2({\bK})^n$ and $f \in C^0({\bK})$ solving
\begin{subequations}\label{eq:poisson}
\begin{align}
 -\Delta u = f &&&\text{ in } \bOm, \label{eq:laplacien} \\
 \partial_\vn u = 0 &&& \text{ on } \partial \bOm, \label{eq:neumann} \\
 f \leq 1 &&& \text{ in } \bOm, \label{eq:sourcin} \\
 f = 1 &&& \text{ on } \partial \bOm, \label{eq:sourceq}
\end{align}
\end{subequations}
then problem \eqref{lp:jean_dual} has an optimal solution.
\end{lem}

This rephra{s}ing is a Poisson PDE \eqref{eq:laplacien} with Neumann boundary condition \eqref{eq:neumann}, whose source term $f$ is a parameter subject to constraints \eqref{eq:sourcin} and \eqref{eq:sourceq}.

\begin{rmk}[Loss of generality] Looking for $\vu$ under the form $\vu = \grad u$ makes us loose the equivalence. Indeed, while \eqref{lp:jean_dual} and \eqref{eq:optimum} are equivalent, existence of a solution to \eqref{eq:poisson} is only a sufficient condition for existence of an optimum for \eqref{lp:jean_dual}, since \eqref{eq:optimum} might have only solutions $\vu$ that are not gradients.
\end{rmk}

\begin{rmk}[Invariant set for gradient flow] From a dynamical systems point of view, the constraint in \eqref{lp:jean_dual} which states 
that the inner product of $\vu=\grad u$ with $\grad g$ is non-positive on $\partial\bOm$, means that we are looking for a velocity field or control $\vu$ in the form of the gradient of a potential $u$ such that ${\bK}$ is an invariant set for the solutions $t \in \R \mapsto \vx(t) \in \R^n$ of the Cauchy problem
$$ \dot{\vx}(t) = -\grad u(\vx(t)),\quad \vx(0) \in \bB$$
after what we just have to define $h := 1 + \Delta u$ on $\bOm$.
\end{rmk}

\subsection{Regular solutions to the Poisson PDE} \label{sec:pdes}

It remains to prove existence of solutions to problem \eqref{eq:poisson}. First, notice that PDE \eqref{eq:laplacien} together with its boundary condition \eqref{eq:neumann} enforces an important constraint on the source term $f$, namely its mean must vanish:
\begin{equation}\label{eq:moy}
\int_\bOm f \; d\lambda = 0.
\end{equation}
Indeed, if $(f,u)$ solves \eqref{eq:poisson}, then 
\begin{align*}
\int_\bOm f \; d\lambda \; & \stackrel{\eqref{eq:laplacien}}{=} - \int_\bOm \Delta u \; d\lambda \\
& \stackrel{\eqref{eq:gaussform}}{=} - \int_{\partial\bOm} \grad u \cdot \vn_\bOm \; d\sigma
 \,\stackrel{\eqref{eq:neumann}}{=} 0\,.
\end{align*}
Moreover, the following holds.

\begin{lem}[Existence { and regularity} on a connected domain] \label{thm:exist}
Suppose that $\bOm$ is connected. Let the source term $f
$ {be Lipschitz continuous on $\bK$ and}
 have zero mean on $\bOm$.  Then there exists $u \in C^
{2}
({\bK})$ satisfying \eqref{eq:laplacien} and \eqref{eq:neumann}.
\end{lem}
\begin{proof}
{
This is a direct application of \cite{schauder}: for $\alpha \in (0,1)$, since $\bK$ is bounded (let $R>0$ be such that $\bK \subset \{\vx \in \R^n : \|\vx\| \leq R \}$) and $f$ is Lipschitz (let $L$ be its Lipschitz constant on $\bK$), one has for $\vx,\vy \in \bK$ that 
$$|f(\vx) - f(\vy)| \leq L \, \|\vx-\vy\| \leq L \, \|\vx-\vy\|^{1-\alpha} \|\vx-\vy\|^\alpha \leq \underbrace{L \, (2R)^{1-\alpha}}_{< \infty} \|\vx-\vy\|^\alpha,$$
so that
$$f \in C^{0,\alpha}(\bK) := \left\{\varphi \in C^0(\bK) : \sup\limits_{\vx,\vy \in \bK} \frac{|\varphi(\vx) - \varphi(\vy)|}{\|\vx-\vy\|^\alpha} < \infty \right\}, $$
and \cite{schauder} yields a solution
$$ u \in C^{2,\alpha}(\bK) := \left\{\varphi \in C^2(\bK) : \sup\limits_{\vx,\vy \in \bK} \frac{\|\mH(\varphi)(\vx) - \mH(\varphi)(\vy)\|}{\|\vx-\vy\|^\alpha} < \infty \right\} $$
to the Poisson PDE \eqref{eq:laplacien} with Neumann boundary condition \eqref{eq:neumann}, where $\mH(\varphi) = \left(\dfrac{\partial^2\varphi}{\partial x_i\partial x_j}\right)_{i,j}$ is the Hessian matrix of $\varphi$.
}
\end{proof}

\begin{rmk}
Assuming that $\partial \bOm$ is $C^\infty$ instead of $C^1$ is actually without loss of generality since $\bOm$ is a semi-algebraic set: as soon as $\partial \bOm$ is locally the graph of a $C^1$ function, it is smooth.
\end{rmk}

In Lemma \ref{thm:exist}, we assumed that $\bOm$ is connected, so that we could apply the 
{results of \cite{schauder}. T}o tackle non-connected sets{, we recall that} $\bOm$ is a semi-algebraic set, {hence} it has a finite number of connected components $\bOm_1,\dots,\bOm_N$. {Moreover, the regularity of $\partial\bOm$ ensures that the $\bOm_i$ have disjoint closures, so that we can trivially compute a solution $u_i$ on each $\bOm_i$ and glue the $u_i$ together into a solution $u := \sum_{i=1}^N \ind_{\overline\bOm_i}u_i$ on the whole $\bOm$.

\begin{rmk}
Tackling the non-connected case requires that $f$ has zero mean on each connected component of $\bOm$:
$$ \forall i \in \{1,\ldots,N\}, \quad \int_{\bOm_i} f \, d\lambda = 0. $$
\end{rmk}}

\begin{rmk}
Lemma \ref{thm:exist} automatically enforces $-\Delta u = 1$ on $\partial \bOm$, which is crucial for the continuity of the optimization variable $w$.
\end{rmk}

\subsection{Explicit optimum for volume computation with Stokes constraints}

Our optimization problem does not feature only the Poisson PDE with Neumann condition: it also includes constraints \eqref{eq:sourcin} and \eqref{eq:sourceq} on the source term. Consequently, a {Lipschitz continuous} function $f$ on $\bK$ 
with zero integral over any connected component of $\bOm$ and satisfying \eqref{eq:sourcin} and \eqref{eq:sourceq} remains to be constructed. We keep the notations of {Section \ref{sec:pdes}} and suggest as candidate
\begin{equation}\label{eq:source}
\vx\mapsto f(\vx) := 1 - g(\vx)\sum_{i=1}^N\frac{\ind_{\bOm_i}(\vx)}{m_{\bOm_i}(g)}.
\end{equation}
By definition, $g = 0$ on $\partial \bOm$, so that \eqref{eq:sourceq} automatically holds. Moreover, both $g$ and $\ind_{\bOm_i}$ are nonnegative on $\bK$, so that \eqref{eq:sourcin} also holds.

In terms of regularity, $f$ {is continuous and piecewise polynomial, so it is Lipschitz continuous on $\bK$}.

Eventually, let $i \in \{1,\ldots,N\}$ so that $\bOm_i$ is a connected component of $\bOm$. Then, by definition, $\partial \bOm_i \subset \partial \bOm$, and one has
\begin{align*}
\int_{\bOm_i} f \; d\lambda \ & = \int_{\bOm_i} \left(1 - g(\vx)\sum_{i=1}^N\frac{\ind_{\bOm_i}(\vx)}{m_{\bOm_i}(g)}\right) \; \md\vx \\
& = \lambda(\bOm_i) - \frac{1}{m_{\bOm_i}(g)} \int_{\bOm_i} g(\vx) \; \md\vx \,=\,0,
\end{align*}
by definition {of} $m_{\bOm_i}(g)
$. {This, together with Lemmata \ref{lem:optim} and \ref{thm:exist}, concludes the proof of Theorem \ref{main}.} 

Indeed, one can check that { for the resulting $w^\star(\vx) = g(\vx)\sum_{i=1}^N\frac{\ind_{\bOm_i}(\vx)}{m_{\bOm_i}(g)}$,}
\begin{align*}
\int_\bB w^\star \; d\lambda \ & = \sum_{i=1}^N \frac{1}{m_{\bOm_i}(g)} \int_{\bOm_i} g \; d\lambda \\
& = \sum_{i=1}^N \lambda(\bOm_i) \, =\, \lambda(\bOm) \, = \,\lambda(\bK).
\end{align*}

\section{Examples}

To illustrate how efficient can be the introduction of Stokes constraints for volume computation, we consider the simple setting where $\bK$ is a Euclidean ball included in $\bB$ the unit Euclidean ball{, as well as some basic variations around this case, where $\bK$ is a non-euclidean ball or a union of balls, or $\bB$ a non-euclidean ball}. Indeed drastic improvements on the convergence are observed. All numerical examples were processed on a standard laptop computer under the Matlab environment with the SOS parser of YALMIP \cite{l04}, the moment parser GloptiPoly \cite{hll09} and the semidefinite programming solver of MOSEK \cite{d12}. {For an interested reader, the codes used to obtain the results presented in this section are available online: \url{https://homepages.laas.fr/henrion/software/stokesvolume/}}

\subsection{Practical implementation} \label{sec:implem}

Following the Moment-SOS hierarchy methodology for volume computation
as described in \cite{hls09}, in the (finite-dimensional) degree $d$ semidefinite {strengthening} of  dual problem \eqref{lp:jean_dual} {with unit euclidean ball as the bounding box $\bB$}:
\begin{itemize}
\item $w\in\R[\vx]_{d}$ and $\vu\in\R[\vx]^n_{d}$ are polynomials of degree at most $d$;
\item the positivity constraint $w \in C^0(\bB)_+$ is replaced with a Putinar certificate 
of positivity on $\bB$, that is:
\[w(\vx)\,=\,\sigma_0(\vx)+\sigma_1(\vx)(1-|x|^2)\,,\quad\forall \vx\in\R^n\,,\]
where $\sigma_0$ (resp. $\sigma_1$) is an SOS polynomial of degree at most $d$ (resp. $d-2$);
\item the positivity constraint $w\vert_\bK - \div \vu - 1 \in C^0(\bK)_+$ is replaced with
a Putinar certificate  of positivity on $\bK$, that is:
\[w(\vx)-\div \vu(\vx) -1\,=\,\psi_0(\vx)+\psi_1(\vx)\,g(\vx)\,,\quad\forall \vx\in\R^n\,,\]
where $\psi_0$ (resp. $\psi_1$) is an SOS polynomial of degree at most $d$ (resp. $d-{\rm deg}(g)$);
\item the positivity constraint $(\vu \cdot \grad g)|_{\partial\bK} \in C^0(\partial \bK)_+$
is replaced with a Putinar certificate  of positivity on $\partial\bK$, that is:
\[ - \vu(\vx) \cdot \grad g(\vx)\,=\,\eta_0(\vx)+\eta_1(\vx)\,g(\vx)\,,\quad\forall \vx\in\R^n\,,\]
where $\eta_0$ is an SOS polynomial of degree at most $d$ 
and $\eta_1$ is a polynomial of degree at most $d-{\rm deg}(g)$;
\item the linear criterion $\int_{\bB} w \, d\lambda$ translates into {a} linear criterion on the vector of coefficients of $w$,
as $\int_{\bB}\vx^\alpha\,d\lambda$ is available in closed-form.
\end{itemize}

The above identities define linear constraints on the coefficients of all the u{n}known polynomials.
Next, stating that some of these polynomials must be SOS 
translate{s} into semidefinite constraints on their respective unknown Gram matrices.
The resulting optimization problem is a semidefinite program{, called the SOS strengthening of problem \eqref{lp:jean_dual}, and is in lagrangian duality with the so-called moment relaxation of problem \eqref{lp:jean_primal}. Using the strong duality property, we interchangeably use the SOS strengthenings and moment relaxations, as they are equivalent}; for more details the interested reader is referred to e.g. \cite{hls09}.

\subsection{Bivariate disk}\label{ex:disk}

Let us first illustrate Theorem \ref{main} for computing the area of the disk $\bK:=\{\vx \in \R^2 : g(\vx) = 1/4 - (x_1-1/2)^2 - x^2_2 \geq 0\}$ included in the unit disk $\bB:=\{\vx \in \R^2 : 1 - x^2_1 - x^2_2 \geq 0\}$.

\begin{figure}[h]\centering
\includegraphics[width=0.495\textwidth]{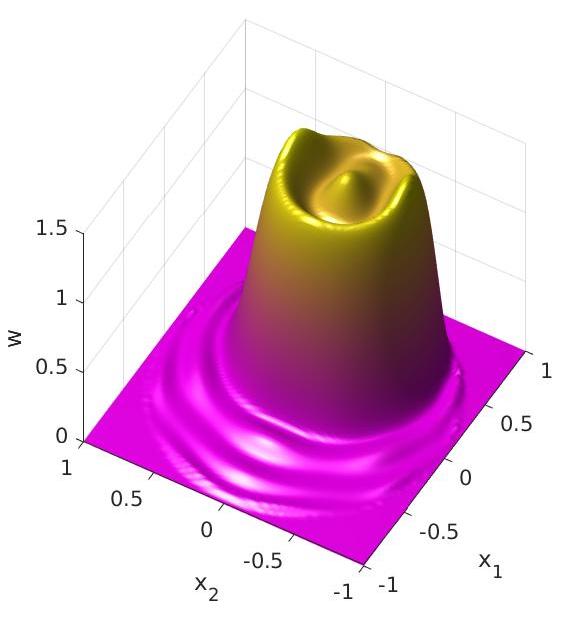}
\includegraphics[width=0.495\textwidth]{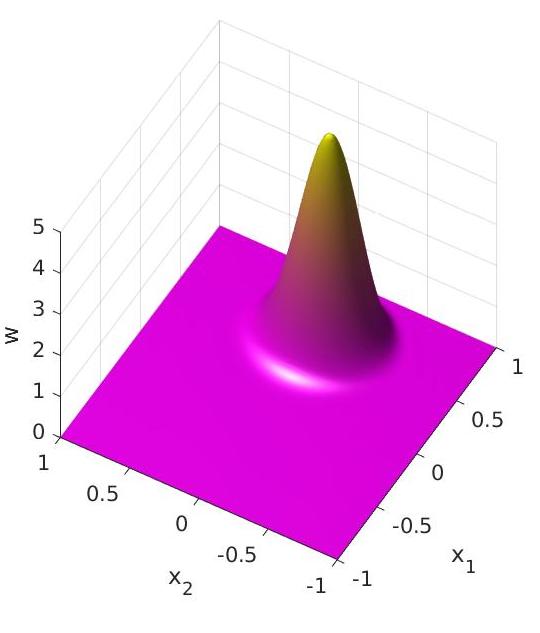}
\caption{Degree 16 polynomial approximations of the disk's area obtained without Stokes constraints (left) and with Stokes constraints (right).\label{fig:disk}}
\end{figure}

The degree $d=16$ polynomial approximation $w$ obtained by solving the SOS {strengthening} of linear problem \eqref{pl:dual} is represented at the left of Figure \ref{fig:disk}. We can see bumps and ripples typical of 
a Gibbs phenomenon, since the polynomial should approximate from above
the discontinuous indicator function $\ind_\bK$  as closely as possible.
A rather loose upper bound of 1.1626 is obtained on the volume $\lambda(\bK)=\frac{\pi}{4}\approx 0.7854$.

In comparison, the degree $d=16$ polynomial approximation $w$ obtained by solving the SOS {strengthening} of linear problem \eqref{lp:jean_dual} is represented at the right of Figure \ref{fig:disk}. As expected from the proof of Theorem \ref{main}, the poynomial should approximate from above the continuous function $g\ind_\bK\:\lambda(\bK)/(\int g\lambda_\bK)$. The resulting polynomial approximation is smoother and yields
a much improved upper bound of 0.7870.

\subsection{Higher dimensions}

In Table \ref{tab:3ball} we report on the dramatic acceleration brought by Stokes constraints in the case of the Euclidean ball $\bK:=\{\vx \in \R^3 : g(\vx) = (3/4)^2 - |\vx|^2 \geq 0\}$ of dimension $n=3$ included in the unit ball $\bB$. We specify the relative errors on the bounds obtained by solving moment relaxations with and without Stokes constraints, together with the computational times (in seconds), for a relaxation degree $d$ ranging from 4 to 20. We observe that tight bounds are obtained already at low degrees with Stokes constraints, sharply contrasting with the loose bounds obtained without Stokes constraints. However, we see also that the inclusion of Stokes constraints has a computational price.

\begin{table}[h]\centering
\begin{tabular}{c|c|c|c}
$n$ & $d$ & without Stokes & with Stokes \\ \hline
3 & 4 & 88\% (0.03s) & 18\% (0.04s) \\
3 & 8 & 57\% (0.16s) & 1.0\% (0.44s) \\
3 & 12 & 47\% (1.97s) & 0.0\% (4.63s) \\
3 & 16 & 43\% (23.9s) & 0.0\% (30.1s) \\
3 & 20 & 41\% (142s) & 0.0\% (206s) \\
\end{tabular}
\caption{Relative errors ($\%$) and computational times (in brackets in seconds) for solving moment relaxations of increasing degrees $d$ approximating the volume of ball of dimension $n=3$.\label{tab:3ball}}
\end{table}

In Table \ref{tab:ball} we report the relative errors on the bounds obtained with and without Stokes constraints, together with the computational times (in seconds), for a relaxation degree equal to $d=10$ (left) resp. $d=4$ (right) and for dimension $n$ ranging from 1 to 5 (left) resp. from 6 to 10 (right). When $d=10$ and $n=5$ the semidefinite relaxation features 6006 pseudo-moments without Stokes constraints, and 12194 pseudo-moments with Stokes constraints.  We see that introducing Stokes constraints incurs a computational cost, to be compromised with the expected quality of the bounds. 

\begin{table}[h]\centering
\begin{tabular}{c|c|c|c}
$n$ & $d$ & without Stokes & with Stokes \\ \hline
1 & 10 & 17\% (0.05s) & 0.0\% (0.03s) \\
2 & 10 & 35\% (0.09s) & 0.2\% (0.25s) \\
3 & 10 & 56\% (0.52s) & 0.3\% (1.19s) \\
4 & 10 & 72\% (9.74s) & 0.4\% (22.8s) \\
5 & 10 & 79\% (150s) & 0.6\% (669s) \\
\end{tabular}
\hspace{2em}
\begin{tabular}{c|c|c|c}
$n$ & $d$ & without Stokes & with Stokes \\ \hline
6 & 4 & 190\% (0.25s) & 45.1\% (1.03s) \\
7 & 4 & 203\% (0.32s) & 60.0\% (4.88s) \\
8 & 4 & 221\% (0.42s) & 78.6\% (8.45s) \\
9 & 4 & 245\% (1.15s) & 102\% (45.1s) \\
10 & 4 & 278\% (3.10s) & 131\% (176s) \\
\end{tabular}
\caption{Relative errors ($\%$) and computational times (in brackets in seconds) for solving the degree $d=10$ (left) and $d=4$ (right) moment relaxation approximating the volume of a ball of increasing dimensions $n$.\label{tab:ball}}
\end{table}

Higher dimensional problems can be addressed only if the problem description has some sparsity structure, as explained in \cite{twlh19}. Also, depending on the geometry of the problem, and for larger values of the relaxation degree, alternative polynomial bases may be preferable numerically than the monomial basis which is used by default in Moment and SOS parsers {(see \cite[Fig.4.5]{hls09})}.

{
\subsection{Changing the bounding box} \label{sec:lp}
Choosing the unit euclidean ball as our bounding box $\bB$ is the easiest and most standard choice, but one could wonder what happens if we take another set, for example an $\ell^p$ ball for $p > 2$. Let $\bB^n_p := \{\vx \in \R^n : \|\vx\|_p^p= \sum_{i=1}^n |x_i|^p \leq 1 \}$  denote the unit $\ell^p$ ball in dimension $n$.
We now compute the area of the bivariate disk $\bK = \{\vx \in \R^2 : (3/4)^2 - x_1^2 - x_2^2 \geq 0\}$ included in the unit $\ell^p$ ball $\bB = \bB^2_p$ for $p = 2,4,6,8,10$. To that end, we use the closed formula for the Lebesgue moments on $\bB^2_p$ (see Appendix \ref{app:lp}) :
$$ \int_{\bB^2_p} \vx^{\vk} \, d\vx = 0 \qquad \forall \vk \in \N^n \setminus (2\N)^n \qquad \text{and}$$
\begin{align*}
\int_{\bB^2_p} \vx^{2\vk} \, d\vx \
& = \frac{2}{(1+|\vk|) \ p} \Beta\left(\frac{1+2k_1}{p} , \frac{1+2k_2}{p}\right)
\end{align*}
where $\Gamma(x) := \displaystyle\int_0^\infty \e^{-t} \ t^{x-1} \, dt$ and $\Beta(x,y) := \dfrac{\Gamma(x)\Gamma(y)}{\Gamma(x+y)}$ are Euler's Gamma and Beta functions, so that in particular one has $\Gamma(1+x) = x\ \Gamma(x)$.\\

\noindent We then perform the computations using Matlab's \texttt{beta} or \texttt{gamma} commands.  Our numerical results are reported in Table \ref{tab:lp}.
\begin{table}[h]\centering
\begin{tabular}{c|c|c|c}
$d$ & $p$ & without Stokes & with Stokes \\ \hline
8 & 2 & 39\% (0.13s) & 0.6\% (0.14s) \\
8 & 4 & 52\% (0.13s) & 1.2\% (0.14s) \\
8 & 6 & 57\% (0.13s) & 1.7\% (0.13s) \\
8 & 8 & 59\% (0.13s) & 2.0\% (0.14s) \\
8 & 10 & 61\% (0.13s) & 2.2\% (0.14s) \\
\end{tabular}
\hspace{2em}
\begin{tabular}{c|c|c|c|c}
$d$ & $p$ & without Stokes & with Stokes \\ \hline
16 & 2 & 27\% (0.47s) & 0.0\% (0.70s) \\
16 & 4 & 33\% (0.41s) & 0.0\% (0.52s) \\
16 & 6 & 36\% (0.33s) & 0.0\% (0.50s) \\
16 & 8 & 37\% (0.30s) & 0.0\% (0.53s) \\
16 & 10 & 38\% (0.35s) & 0.0\% (0.49s) \\
\end{tabular}
\caption{Relative errors ($\%$) and computational times (in brackets in seconds) for solving the degree $d=8$ (left) and $d=16$ (right) moment relaxations approximating the volume of ball of dimension $n=2$ embedded in an $\ell^p$ ball bounding box for $p=2,4,6,8,10$.\label{tab:lp}}
\end{table}

Unsurprisingly, as in the case of the euclidean bounding box, Stokes constraints drastically improve the accuracy of the moment relaxations.  However, the number $p$ has an influence on both accuracy (decreasing with $p$) and computational time (global tendency to slightly decrease with $p$) in the original as well as Stokes-augmented hierarchies. This can be explained by analyzing the influence of $p$ on the SOS strengthenings described in section \ref{sec:implem}: indeed, the only change is that $1 - |\vx|^2$ is replaced with $1 - \|\vx\|_p^p$ in the SOS representation of constraint $w \in \cC(\bB)_+$, so that $\sigma_1$ now has degree at most $d-p$ instead of $d-2$. Consequently, with increasing $p$, the size of the Gram matrix of $\sigma_1$ becomes smaller, slightly reducing the size of the corresponding SDP problem, which can result in a reduction of the computational time (although other factors impact the computational time, hence a non monotonic function of $p$). Conversely, as the degree of $\sigma_1$ is more limited, this brings less freedom in the search for an optimal solution, hence reducing the accuracy of the SOS strengthening for a fixed degree $d$. 

In terms of the efficiency of Stokes constraints when the bounding box is an $\ell^p$ ball, we highlight the fact that the loss in accuracy is less important in the Stokes-augmented hierarchy than in the standard formulation: Stokes constraints are somewhat more robust to the increase in the degree of the polynomial describing the bounding box. Regarding computational times, when they decrease with $p$,  we observe that this decrease is more important with Stokes constraints than without: here again, Stokes constraints lead to a better behaved hierarchy.

\subsection{Other sets}

So far we only computed the volume of euclidean balls. In order to further explore the influence of the input polynomials on the efficiency of Stokes constraints as well as the Moment-SOS hierarchy in general, we now switch to computing the volume of more sophisticated semi-algebraic sets in two dimensions: first, we proceed from the euclidean ball to generic $\ell^p$ balls, as we did with the bounding box; second, we test the limits of the scheme by approximating the volume of a non-convex, non-connected double disk.

\subsubsection{$\ell^4$ disk} \label{sec:lpdisk}

As discussed in section \ref{sec:lp}, the degree of the polynomials involved in the Moment-SOS hierarchy has a direct influence on its accuracy, as a higher input degree means, for a fixed degree of the hierarchy, less degrees of freedom to optimize over. We now show on one example what the practical influence of the degree over our scheme's accuracy is, by computing the approximate volume of the $\ell^4$ disk:

$$\bK := \left\{\vx \in \R^2 : \left(25/72\right)^4 - x_1^4 - x_2^4 \geq 0 \right\}. $$

It is clear that one has $\lambda(\bK) =  \left(25/72\right)^2 \lambda(\bB_4^2)$ with, using formula \eqref{eq:mompball} from Appendix \ref{app:lp}, $\lambda(\bB_4^2) =
\frac{\Gamma(1/4)^2}{2\sqrt{\pi}}$
so that $\lambda(\bK) \approx 0.4471$. We implement the degree 16 SOS strengthenings corresponding to the standard and Stokes-augmented problems, and plot the resulting $w$ in Figure \ref{fig:lpdisk}.

\begin{figure}[h]  \centering
\includegraphics[width=0.495\textwidth]{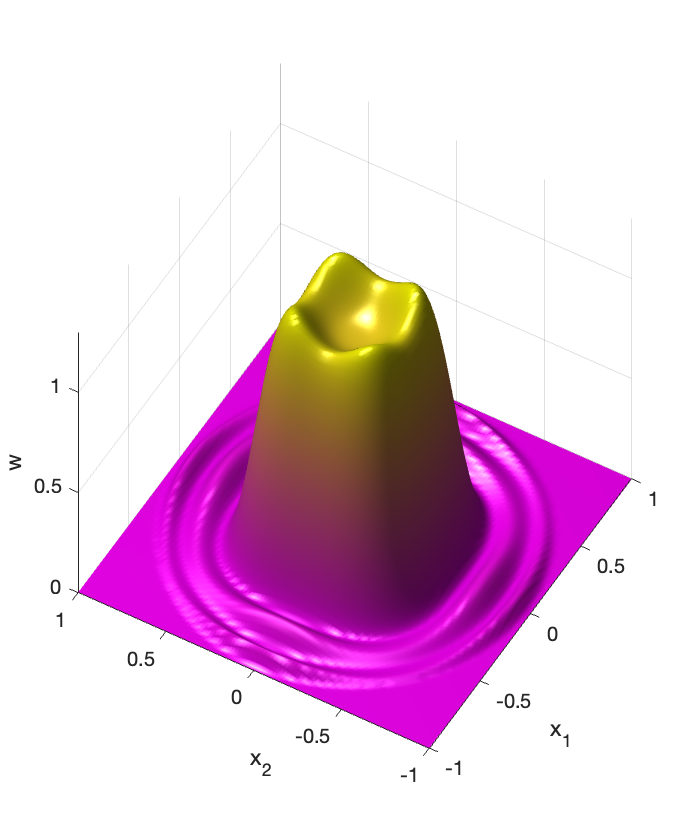}
\includegraphics[width=0.495\textwidth]{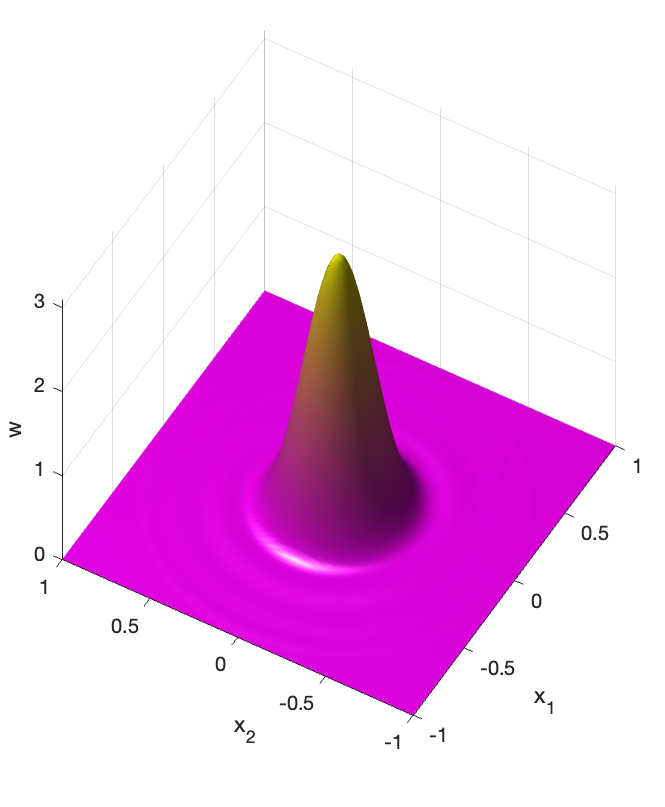}
\caption{Degree 16 polynomial approximations of the area of the $\ell^4$ disk obtained without Stokes constraints (left) and with Stokes constraints (right).\label{fig:lpdisk}}
\end{figure}

Again, the original SOS strengthening is flawed by a strong Gibbs phenomenon that introduces a large error in the volume approximation (we get a bound of $0.8511$, i.e. a relative error of $90 \%$), characterized by wide oscillations on the boundary of the $\ell^4$ disk $\bK$. The Stokes-augmented version gives a tighter bound of $0.4653$ (relative error $4 \%$, still more than for the euclidean disk, but much less than without Stokes constraints). Moreover, an interesting feature appears here that was not visible on Figure \ref{fig:disk} in the case of the euclidean disk: we observe small oscillations of $w$ around $0$ on $\bB \setminus \bK$. This can be expected as $w$ is a non-zero polynomial, so it cannot vanish on a set of positive Lebesgue measure.  These observations confirm our predictions that the lower the degree of the involved polynomials, the more accurate the SDP relaxations. However, we are now going to show that some other parameters should be considered when discussing the accuracy of the Moment-SOS hierarchy for volume computation, such as the geometry of the considered set $\bK$.

\subsubsection{Disconnected double disk} \label{sec:ddisk}

We finally test our numerical scheme on a non-convex, non-connected semi-algebraic set:
$$ \bK := \left\{\vx \in \R^2 : \left(1/16 - (x_1 - 1/2)^2 - x_2^2\right)\left((x_1 + 1/2)^2 + x_2^2 - 1/16\right) \geq 0 \right\}.$$

\begin{figure}[h]  \centering
\includegraphics[width=0.495\textwidth]{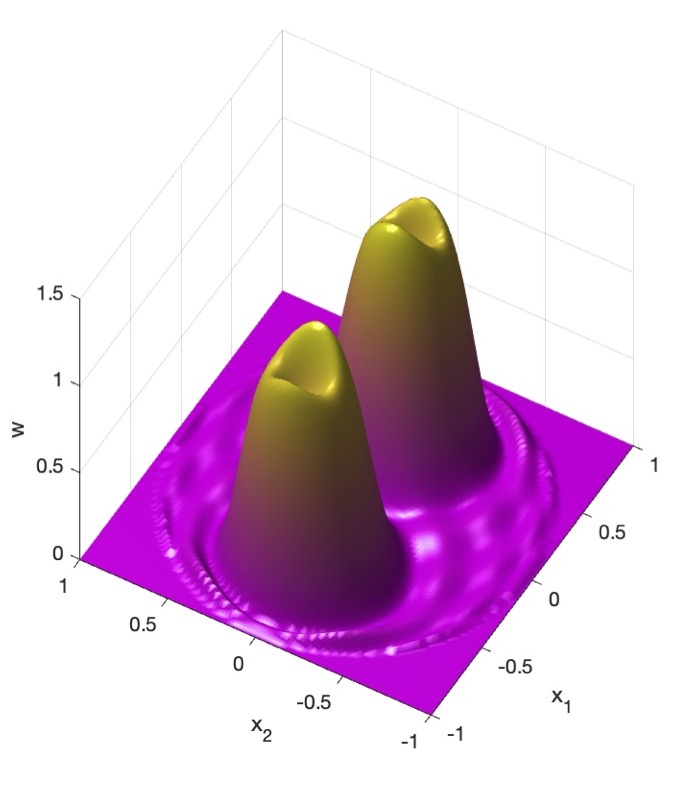}
\includegraphics[width=0.495\textwidth]{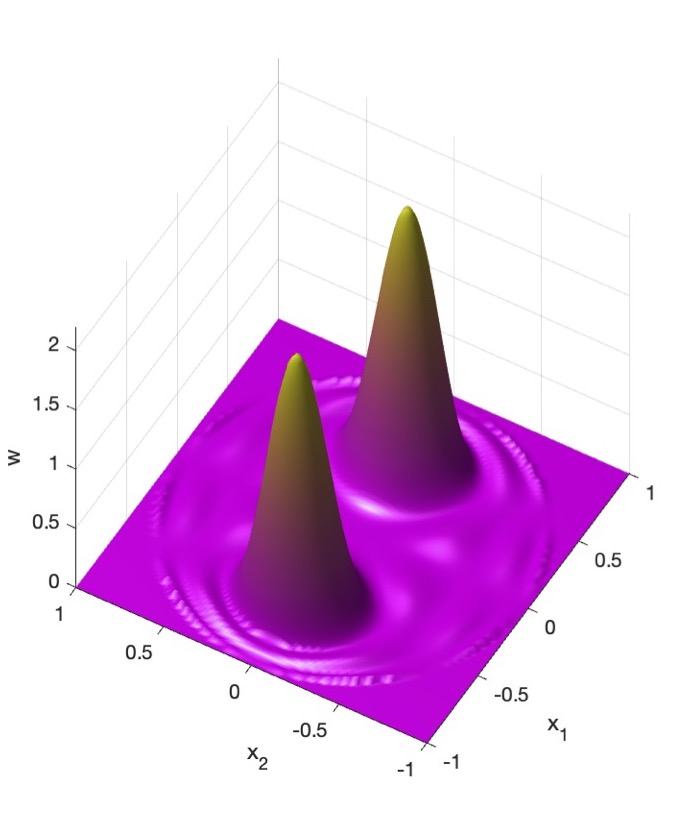}
\caption{Degree 16 polynomial approximations of the area of the double disk obtained without Stokes constraints (left) and with Stokes constraints (right).\label{fig:doubledisk}}
\end{figure}

As usual, on Figure \ref{fig:doubledisk}  we observe a Gibbs phenomenon in the standard volume approximation scheme (with a bound of $0.8551$ instead of $\frac{\pi}{8} \approx 0.3927$, i.e. a relative error of $118 \%$), as well as wide oscillations near the boundary of $\bK$. As for the Stokes-augmented scheme, again we get a better bound of $0.4671$ (relative error $19 \%$, interestingly higher than in all the previous cases with Stokes constraints, but still much more accurate than without Stokes constraints). More striking here, even in the Stokes-augmented SOS strengthening, $w$ is seen clearly oscillating. However, this should not be mistaken for a consequence of the Gibbs phenomenon, as in this new formulation $w$ is proved to approximate a Lipschitz continuous function. As a consequence to the Stone-Weierstrass theorem, those remaining oscillations are bound to ultimately vanish as the degree $d$ goes to infinity, while in the case of the Gibbs phenomenon, the oscillations do not ultimately disappear (only their contribution to $\int w \; d\lambda$ ultimately vanishes).  

A possible explanation for this oscillatory phenomenon is that, despite the regularity of the optimizer in the infinite dimensional problem \eqref{lp:jean_dual}, the SOS strengthenings are still very demanding for the polynomial $w$: indeed, it is requested to be as close to $0$ as possible outside $\bK$ while being sufficiently large in $\bK$ so that its integral is bigger than $\lambda(\bK)$. In the case of a non-connected $\bK$, $w$ is thus literally requested to oscillate. 

To conclude on this example, we highlight the fact that, in addition to the degree of the polynomial $g$ defining $\bK$, the geometry of $\bK$ (typically: its number of connected components and how they are distributed in the bounding box $\bB$) plays a key role in the accuracy of the Moment-SOS hierarchy for computing its volume, both in the original and Stokes-augmented versions. Indeed, it is this geometry that is likely to generate (or, on the contrary, prevent) an oscillatory behavior in the approximating polynomial $w$, when one gets rid of the Gibbs phenomenon by complementing the hierarchy with Stokes constraints. 
This is particularly visible when comparing our examples in Sections \ref{sec:lpdisk} and \ref{sec:ddisk}, where $\bK$ is described by degree $4$ polynomials, but the schemes are far more accurate in the convex case than in the disconnected case, especially when one adds Stokes constraints.
}

\section{Conclusion}

In this paper we proposed a new primal-dual infinite-dimensional linear formulation of the problem of computing the volume of a smooth semi-algebraic set generated by a single polynomial, generalizing the approach of \cite{hls09} while still allowing the application of the Moment-SOS hierarchy. The new dual formulation contains redundant linear constraints arising from Stokes's Theorem, generalizing the heuristic of \cite{l17}. A striking property of this new formulation is that the dual value is attained, contrary to the original formulation. As a consequence, the corresponding dual SOS hierarchy does not suffer from the Gibbs phenomenon, thereby accelerating the convergence.

Numerical experiments (not reported here) reveal that the values obtained with the new Stokes constraints (with a general vector field) are closely matching the values obtained with the original Stokes constraints of \cite{l17} (with the generating polynomial factoring the vector field). It may be then expected that the original and new Stokes constraints are equivalent. However at this stage we have not been able to prove equivalence.

The proof of dual attainment builds upon classical tools from linear PDE analysis, thereby building up a new bridge between infinite-dimensional convex duality and PDE theory, in the context of the Moment-SOS hierarchy. We expect that these ideas can be exploited to prove regularity properties of linear reformulations of other problems in data science, beyond volume approximation. For example, it would be desirable to design Stokes constraints tailored to the infinite-dimensional linear reformulation of the region of attraction problem \cite{hk14} or its sparse version \cite{tchl20}.

{ In terms of practical implementation, while still observed with Stokes constraints, the dependence on the degree of the input polynomials, already discussed in \cite{hls09}, seems to be of less importance.  However, the dependence in the geometry of $\bK$ now seems to prevail, as Stokes constraints add information on this geometry; more precisely, the simpler the geometry, the more efficient the constraints: the smooth and convex case leads to the best increase in accuracy, but the dual attainment still holds even on disconnected smooth sets. Also, experiments carried out in \cite{l17,twlh19} show that even in the non-smooth case, Stokes constraints drastically improve the accuracy of the volume computing scheme.}

\appendix

{

\section{Moments of the Lebesgue measure on an $\ell^p$ ball} \label{app:lp}

For our numerical experiments, we use closed formulae for the moments of the Lebesgue measure on an $\ell^p$ ball. These formulae can be derived in a quite straightforward fashion using \cite[\sc Theorem 2.2]{l16}. 
\begin{dfn}[Positively homogeneous functions]
For $d \in \R$, $h : \R^n \longrightarrow \R_+$ is said to be {positively homogeneous of degree} $d$ if for all $\vx \in \R^n \setminus \{\vO\}$, $\lambda > 0$, one has
$$ h(\lambda \vx) = \lambda^d h(\vx). $$
\end{dfn}
\begin{lem}[Lebesgue moments and positively homogeneous functions, {\sc Theorem 2.2} in \cite{l16}]\label{lem:momhf}
Let $h : \R^n \longrightarrow \R$ be a positively homogeneous function of degree $d \in \R \setminus \{0\}$. Let $\bB_h := \{\vx \in \R^n : h(\vx) \leq 1 \}$ be such that $\lambda(\bB_h) < \infty$. Then, for all $\vk \in \N^n$,
\begin{equation}\label{eq:momhf}
\int_{\bB_h} \vx^\vk \, d\vx = \Gamma\left(1+\frac{n+|\vk|}{d}\right)^{-1} \int_{\R^n} \vx^\vk \, \e^{-h(\vx)} \, d\vx.
\end{equation}
\end{lem}
From this we deduce the following closed formula for moments of the Lebesgue measure on an $\ell^p$ ball:
\begin{lem}[Lebesgue moments on the unit $\ell^p$ ball] \label{lem:mompball}
Let $p > 1$. The even moments of the Lebesgue measure on $\bB^n_p$ are given, for $\vk = (k_1,\ldots,k_n) \in (2\N)^n$, by
\begin{equation}\label{eq:mompball}
\int_{\bB^n_p} \vx^\vk \, d\vx = \frac{2^n}{p^n} \Gamma\left(1+\frac{n+|\vk|}{p}\right)^{-1}\prod\limits_{i=1}^n\Gamma\left(\frac{1+k_i}{p}\right).
\end{equation}
The odd moments are $\int_{\bB^n_p} \vx^\vk \, d\vx = 0$ for $\vk \in \N^n \setminus (2\N)^n$.
\end{lem}
\begin{proof}
Let $\vk \in \N^n$. Since $\vx \mapsto \|\vx\|_p^p$ is positively homogeneous of degree $p$, we can use Lemma \ref{lem:momhf}:
\begin{equation} \label{apeq:a} \tag{a}
\int_{\bB^n_p} \vx^\vk \, d\vx = \Gamma\left(1+\frac{n+|\vk|}{p}\right)^{-1} \int_{\R^n} \vx^\vk \, \e^{-\|\vx\|_p^p} \, d\vx. 
\end{equation}
Then, Fubini's theorem ensures that 
\begin{equation} \label{apeq:b} \tag{b}
\int_{\R^n} \vx^\vk \, \e^{-\|\vx\|_p^p} \, d\vx = \prod\limits_{i=1}^n \int_{-\infty}^{+\infty} x_i^{k_i} \, \e^{-|x_i|^p}\,dx_i.
\end{equation}
Then, if one of the $k_i$ is odd, then the corresponding $x_i \mapsto x_i^{k_i}\,\e^{-|x_i|^p}$ is an odd function, so that its integral over $\R$ is $0$. Thus, this yields that $\int_{\bB_p^n} \vx^{\vk} \, d\vx = 0$. 
Finally, we compute for $k \in 2\N$
\begin{align}
\int_{-\infty}^{+\infty}x^k \, \e^{-|x|^p} \, dx \ & \stackrel{\text{\ref{lem:momhf}}}{=} \Gamma\left(1+\frac{1+k}{p}\right) \int_{-1}^1 x^k \, dx \nonumber \\
& = \Gamma\left(1+\frac{1+k}{p}\right) \, \frac{2}{k+1} \nonumber \\
& = \frac{2}{\cancel{k+1}} \, \frac{\cancel{1+k}}{p} \, \Gamma\left(\frac{1+k}{p}\right) \label{apeq:c} \tag{c}
\end{align}
where we used the factorial property $\Gamma(1+x) = x \, \Gamma(x)$. \eqref{eq:mompball} is then obtained by combining equations \eqref{apeq:a}, \eqref{apeq:b} and \eqref{apeq:c}.
\end{proof}
}

\end{document}